\documentclass[11pt]{amsart}
\usepackage{amscd,amssymb,graphics}
\usepackage{hyperref}
\usepackage{amsfonts}
\usepackage{amsmath}
\usepackage{enumerate}
\input xy
\xyoption{all}

\newtheorem{thm}{Theorem}[section]

\newtheorem{corol}[thm]{Corollary}
\newtheorem{lemma}[thm]{Lemma}

\theoremstyle{definition}
\newtheorem{defi}[thm]{Definition}

\theoremstyle{remark}
\newtheorem{remark}[thm]{Remark}
\newtheorem{example}[thm]{Example}

\newcommand{\ben}{\begin{enumerate}}
	\newcommand{\een}{\end{enumerate}}
\newcommand{\bit}{\begin{itemize}}
	\newcommand{\eit}{\end{itemize}}

\def\R {{\Bbb R}}

\def\N{{\Bbb N}}

\def\eps{{\varepsilon}}

\def\QED{\nobreak\quad\ifmmode\roman{Q.E.D.}\else{\rm Q.E.D.}\fi}

\def\sna{\scriptscriptstyle\mathcal{NA}}

\begin{document}

\title[]{Balanced and functionally balanced $P$-groups}

\author[]{Menachem Shlossberg}
\email{menachem.shlossberg@uniud.it}

\address{\hfill\break
	Dipartimento di Matematica e Informatica
	\hfill\break
	Universit\`{a} di Udine
	\hfill\break
	Via delle Scienze  206, 33100 Udine
	\hfill\break
	Italy}


	\begin{abstract}
	In relation to  Itzkowitz's problem \cite{Itz91}, we show that a $\mathfrak c$-bounded $P$-group is balanced if and only if it is functionally balanced. We  prove that for an arbitrary $P$-group, being functionally balanced is equivalent to being strongly functionally balanced. A special focus is given to the uniform free topological group defined over a uniform $P$-space. In particular, we show that this group is (functionally) balanced precisely when its subsets $B_n,$ consisting of words of length at most $n,$ are all (resp., functionally) balanced.
	\end{abstract}
	\maketitle
	\section{Introduction and preliminaries}
	A topological group $G$ is {\it balanced} if its left and right uniformities coincide. Recall  that the left uniformity $\mathcal L_G$ of a topological group $G$ is formed by the sets $U_L:=\{(x,y)\in G^2: x^{-1}y\in U\},$ where $U$ is a neighborhood of the identity element of $G.$ The right uniformity $\mathcal R_G$ is defined analogously.
	A topological group $G$ is called {\it functionally balanced} \cite{PR91} in case every bounded left-uniformly continuous real-valued function on $G$ is also right-uniformly continuous. Omitting the term ``bounded" we obtain the definition of a {\it strongly functionally balanced} group. In the sequel we extend these definitions, in a natural way, to include  also the symmetric subsets of a topological group (see Definition \ref{def:bal}).  The question of whether every strongly functionally balanced group is balanced was raised by Itzkowitz \cite{Itz91}.   This longstanding problem is still open.
	
	 Nevertheless, it is known that a functionally balanced group is balanced  whenever $G$ is either  locally compact \cite{Itz91,MIL90,PR91}, metrizable \cite{PR91} or locally connected \cite{MNP97}. Recall that a topological group is {\it non-archimedean} if it has a local base at the identity consisting of open subgroups.
	A strongly functionally balanced non-archimedean group is balanced in case it is $\aleph_0$-bounded \cite{Her2000}  or strongly functionally generated by the set of all its subspaces of countable $o$-tightness \cite{TR04}. For more known results concerning  Itzkowitz's problem we refer the reader to the survey paper \cite{BT07}.
	
	The class of all non-archimedean groups contains the class of all $P$-groups  (see Definition \ref{def:pgr}). We prove that a $P$-group is functionally balanced if and only if it is strongly functionally balanced (Corollary \ref{cor:funbal}). This gives a positive answer to \cite[Question 3]{BT07} for $P$-groups. One of the main result we obtain is that a 
	$\mathfrak c$-bounded  $P$-group  is  balanced  if and only if it is functionally balanced (Theorem \ref{thm:pbal}). So, a negative solution to Itzkowitz's problem cannot be found in the class of $\mathfrak c$-bounded $P$-groups.  
	
	A uniform space that is closed under countable intersection is called a  {\it uniform $P$-space} (see also Definition \ref{def:pgr}). Such a  space is necessarily {\it non-archimedean},  which means that it possesses a    base of equivalence relations (Lemma \ref{lem:pisna}).  In Section \ref{sec:coi} we discuss the coincidence of some universal free objects over the same uniform $P$-space.

	For a free group $F(X)$, over a nonempty set $X$, we denote by $B_n$ its subset containing all words of length not greater than $n.$ In Section \ref{sec:last} we show that the uniform free topological group
	 $F(X,\mathcal U),$ over a uniform $P$-space,  is (functionally) balanced if and only if  $B_n$ is (resp., functionally) balanced for every $n\in \N$. Hopefully, this theorem  can be useful in providing a negative solution to Itzkowitz's problem.
	 
	 Given a symmetric  subset $B$ of a topological group $G$   we denote by 
	  $\mathcal L_G^B$ the trace of the left uniformity  $\mathcal L_G$ on $B.$ That is, $\eps\in  \mathcal L_G^B$ if and only if there exists $\delta\in \mathcal L_G$  such that $\delta\cap(B\times B)= \eps.$ The uniformity   $\mathcal R_G^B$ is  the trace of  $\mathcal R_G$.  In case $\{A_n\}_{n\in \N}$ is a countable collection of subsets of $G,$ we write $\mathcal L_G^n \ (\mathcal R_G^n)$ instead of $\mathcal L_G^{A_n}$ (resp., $\mathcal R_G^{A_n}$). The {\it character} of $G$ is the minimum cardinal of a local base at the identity. For a uniform space $(X,\mathcal U)$, the {\it weight}  $w(X,\mathcal U)$ denotes the minimal cardinality of a base of $(X,\mathcal U).$ For $\eps\in \mathcal U$ and $a\in X$ we let $\eps[a]:=\{x\in X: (a,x)\in \eps\}.$ All topological groups and uniform spaces in this
	  paper are assumed to be Hausdorff. Unless otherwise is stated the uniformity of a topological group $G$ is the two-sided uniformity, that is, the supremum $\mathcal L_G\vee \mathcal R_G.$
Finally, $\mathbf{TGr}, \mathbf{NA}$	and $\mathbf{NA_b}$  denote, respectively, the classes of all topological groups, non-archimedean groups and non-archimedean balanced groups.
	\section{$P$-groups and uniform $P$-spaces}
	\begin{defi}(see \cite{AT}, for example)\label{def:pgr}
		A   {\it $P$-space} is a topological space in which the intersection of countably many open sets is still open. A topological group which is a $P$-space is called 
	a   {\it $P$-group}.  A  {\it uniform $P$-space} $(X,\mathcal U)$ is a uniform space in which the intersection of countably many elements of $\mathcal U$   is again in $\mathcal U$.
	\end{defi}
	\begin{lemma}\cite[Lemma 4.4.1.a]{AT}\label{lem:na}
If $G$ is	a $P$-group, then $G$ is non-archimedean.
	\end{lemma}
	\begin{proof}
		Let $U$ be a neighborhood of the identity element $e.$ We have to show that $U$ contains an open subgroup $H.$
		For every $n\in \N$ there exists a symmetric neighborhood $W_n$ such that $W_n^n\subseteq U.$ Since $G$ is	a $P$-group the set $W=\cap_{n\in \N}W_n$ is a neighborhood of $e.$
		Let $H$  be the subgroup generated by $W.$ Clearly, $H$ is open and $H\subseteq U.$
		\end{proof}
	\begin{lemma}\label{lem:pisna} If $(X,\mathcal U)$ is a uniform $P$-space, then  $(X,\mathcal U)$ is non-archimedean.

	\end{lemma} 
		\begin{proof}
			Let $\eps\in \mathcal U.$ We will  find an equivalence relation $\delta\in  \mathcal U$ such that $\delta \subseteq \eps.$ 	For every $n\in \N$ there exists a symmetric entourage $\delta_n\in \mathcal U$ such that $\delta_n^n\subseteq \eps.$  Since $(X,\mathcal U)$ is a uniform $P$-space, the equivalence relation $\delta=\bigcup_{m\in\N}(\cap_{n\in \N}\delta_n)^m\subseteq \eps$ is an element of $\mathcal U.$ 
			\end{proof}
		\begin{defi} (see \cite{AT,NT}, for example) Let $\tau$ be an infinite cardinal. 
	\ben \item	A topological group $G$ is called {\it $\tau$-bounded} if for every neighborhood $U$ of the identity,  there exists a set $F$ of cardinality not greater than $\tau$ such that 
		$FU=G.$
	\item A uniform space $(X,\mathcal U)$ is {\it $\tau$-narrow } if for every $\eps\in \mathcal U,$ there exists a set $\{x_\alpha: \alpha<\tau\} $ such that $X=\bigcup_{\alpha<\tau}\eps[x_\alpha].$\een
		\end{defi}
		\begin{lemma}\label{lem:bala}
	Let $\tau$ be an infinite cardinal.		Let $G$ be a topological group in which 
	the intersection of any family of cardinality at most $\tau$ of open sets is open. If $G$ is also $\tau$-bounded, then $G$ is balanced.
	\end{lemma}
	\begin{proof}
	 Let $H$ be an open subgroup of $G.$ We will show that there exists a normal open subgroup $N$ of $G$ such that $N\subseteq H.$ 
	Let $N=\cap_{x\in G}xHx^{-1}.$
	Clearly, $N$ is a normal subgroup of $G$ and $N\leq H.$
	We show that $N$ is open.
	Since $G$ is $\tau$-bounded, there exists a subset $F\subseteq G$ with $|F|\leq \tau$ such that $FH=G.$
	It is easy to see that 
	$N=\cap_{x\in F}xHx^{-1}.$ Since
			$|F|\leq \tau,$ this intersection must be open and applying Lemma \ref{lem:na} completes the proof.
		\end{proof}
		\begin{corol}\cite[Lemma 4.4.1.b]{AT}
An
$\aleph_0$-bounded $P$-group is balanced.	\end{corol}
\begin{lemma}\label{lem:punif}
Let $(X,\mathcal U)$ be a uniform $P$-space. A function $f:(X,\mathcal U)\to \R$ is uniformly continuous if and only if there exists $\eps\in \mathcal U$ such that $(x,y)\in \eps \Rightarrow f(x)=f(y).$
\end{lemma}
\begin{proof}
	The ``if" part is trivial for every uniform space (even if it is not a uniform $P$-space). We prove the ``only if" part. Since $f:(X,\mathcal U)\to \R$ is uniformly continuous, for every $n\in \N$ there exists $\eps_n\in \mathcal U$ such that $(x,y)\in \eps \Rightarrow |f(x)-f(y)|<\frac{1}{n}.$ Let $\eps=\cap_{n\in \N}\eps_n.$  Since $(X,\mathcal U)$ is a uniform $P$-space, we have $\eps\in \mathcal U.$ Now,  $(x,y)\in \eps$  implies that $f(x)=f(y).$ 
	\end{proof}
\begin{corol}\label{cor:lco}
Let $G$ be a $P$-group with a symmetric subset $B$. 
A function $f:B\to \R$ is $\mathcal L_G^B$-($\mathcal R_G^B$-)uniformly continuous if and only if there exists an open subgroup $H$ of $G$
such that $f$ is constant on $B\cap xH$ 
(resp., $B\cap Hx$) for every $x\in B$.
\end{corol}
\begin{proof}
Clearly, $(B,\mathcal L_G^B)$ ($(B,\mathcal R_G^B)$) is a uniform $P$-space. Now use Lemma \ref{lem:punif} and the definition of the left  (resp., right)  uniformity to conclude the proof.
\end{proof}
\begin{defi}\label{def:bal}
	We say that a symmetric subset $B$  of a topological group $G$ is:\ben \item  {\it  balanced} if the left and right uniformities of $G$ coincide on $B.$
	\item
	{\it functionally balanced} if every bounded $\mathcal L_G^B$-uniformly continuous function $f:B\to \R$  is $\mathcal R_G^B$-uniformly continuous.
	\item  {\it strongly functionally balanced}  if every  $\mathcal L_G^B$-uniformly continuous function $f:B\to \R$  is $\mathcal R_G^B$-uniformly continuous.	\een\end{defi}
\begin{thm}\label{thm:equiv}
 Let $G$ be a  $P$-group with a symmetric subset $B$.  Then  the following assertions are equivalent:
 \ben \item $B$ is strongly functionally balanced. \item $B$ is functionally balanced.
 \item If 
 $\eps\in  \mathcal L_G^B$ is an equivalence relation with at most
 $\mathfrak c$ equivalence classes, then $\eps\in \mathcal R_G^B.$  
  \een
\end{thm}
\begin{proof}
$(1) \Rightarrow (2):$ Trivial.\\
$(2) \Rightarrow (3):$  Let 
$\eps\in \mathcal L_G^B$ be an equivalence relation with at most
$\mathfrak c$ equivalence classes. It follows that there exists a function $f:B\to [0,1]\subseteq \R$ such that 
$f(x)=f(y) $ if and only if $(x,y)\in\eps.$ Clearly, $f:B\to \R$ is a bounded $\mathcal L_G^B$-uniformly continuous function. Since   $B$ is   functionally balanced, then $f:B\to \R$  is also $\mathcal R_G^B$-continuous. By Corollary \ref{cor:lco}, there exists an open subgroup $H$ of $G$
such that $f$ is constant on $B\cap Hx$ for every $x\in B.$ 
 By the definition of the right uniformity,  $\{(t,s)\in B^2| \ ts^{-1}\in H\}\in \mathcal R_G^B.$ The definition of $f$ implies that $\{(t,s)\in B^2| \ ts^{-1}\in H\}\subseteq \eps,$ and thus $\eps\in \mathcal R_G^B.$\\$(3) \Rightarrow (1):$ Let  $f:B\to \R$  be a $\mathcal L_G^B$-uniformly continuous function. By Corollary \ref{cor:lco}, there exists an open subgroup $H$
such that $f$ is constant on $B\cap xH$ 
for every $x\in B$.
We have $\{(t,s)\in B^2| \ t^{-1}s\in H\}\subseteq \eps:= \{(t,s)| \ f(t)=f(s)\}.$ Hence, $\eps\in\mathcal L_G^B$  and clearly $\eps$ has at most  $\mathfrak c$ equivalence classes. By $(3), \ \eps= \{(t,s)| \ f(t)=f(s)\}\in \mathcal R_G^B.$ Therefore, $f$ is $\mathcal R_G^B$-uniformly continuous and we conclude that $B$ is strongly functionally balanced.
\end{proof}

Letting $B=G$ in Theorem \ref{thm:equiv} we obtain the following:
\begin{corol}\label{cor:funbal}
Let $G$ be a  $P$-group.  Then  $G$ is  functionally balanced if and only if it is strongly functionally balanced.
\end{corol}
Recall the following result of Hern\'{a}ndez:
\begin{thm}\cite[Theorem 2]{Her2000}
Let $G$ be a non-archimedean $\aleph_0$-bounded topological group. Then $G$ is balanced if and only if it is strongly functionally balanced.
\end{thm}

In case  the non-archimedean group is  a $P$-group it suffices to require  $\mathfrak c$-boundedness, as it follows from the following theorem:
\begin{thm}\label{thm:pbal}
Let $G$ be a $\mathfrak c$-bounded $P$-group. Then $G$ is balanced if and only if it is   functionally balanced. 
\end{thm}
\begin{proof}
If $H$ is a subgroup of index at most $\mathfrak c,$ then $\eps:=\{(t,s)| \ t^{-1}s\in H\}$ has at most $\mathfrak c$ equivalence classes.  So, in case $G$ is a $\mathfrak c$-bounded $P$-group, condition $(3)$ of Theorem \ref{thm:equiv}  is equivalent to the coincidence of the left and right uniformities. This completes the proof.
\end{proof}
\begin{remark}\
	\ben [(a)] \item
Theorem \ref{thm:pbal} means that a negative solution to Itzkowitz's problem cannot be found in the class of $\mathfrak c$-bounded $P$-groups. 
\item  Theorem \ref{thm:pbal} can  be viewed also as a corollary of Theorem \ref{thm:equiv}. The latter plays an important role in proving Theorem \ref{thm:last}. As pointed out by the referee, Theorem \ref{thm:pbal} admits a much shorter proof. Indeed, one can take an open subgroup $U$ of $G$ and a map $f:G \to [0,1],$ so that $f$ is constant on each $xU$ and injective on distinct cosets. Since $f$ is left uniformly continuous and $G$ is a functionally balanced $P$-group, there is an open subgroup $V$ such that $f$ is constant on each $Vx.$ Then $Vx$ is contained in $xU$ and $G$ is balanced.\een
\end{remark}
\section{Coincidence of free objects}\label{sec:coi} \begin{defi}\cite[Definition 3.1]{MS13} \label{d:FreeGr} Let $\Omega$ be a subclass of $\mathbf{TGr}$ 
	and 
	$(X,{\mathcal U})$ be a uniform space. By an
	{\it $\Omega$-free topological group} of $(X,{\mathcal U})$ we mean
	a pair $(F_{\Omega}(X,{\mathcal U}),i)$, where
	$F_{\Omega}(X,{\mathcal U})\in \Omega$ and $i: X \to
	F_{\Omega}(X,{\mathcal U})$ is a uniform map satisfying the
	following universal property.   For every uniformly
	continuous map $\varphi: (X,{\mathcal U}) \to G,$ where $G \in \Omega,$ there exists a unique 
	continuous homomorphism $\Phi: F_{\Omega}(X,{\mathcal U})
	\to G $ for which the following diagram commutes:
	\begin{equation*} \label{equ:ufn}
	\xymatrix { (X,{\mathcal U}) \ar[dr]_{\varphi} \ar[r]^{i} &
		F_{\Omega}(X,{\mathcal U})
		\ar[d]^{\Phi} \\
		& G }
	\end{equation*}
\end{defi}
For $\Omega=\mathbf{TGr}$
the universal object
$F_{\Omega}(X,{\mathcal U})$  is the {\it uniform free topological
	group} of $(X,{\mathcal U}).$  This group was invented by Nakayama \cite{N43}
and studied, among others, by Numella \cite{NU82} and Pestov \cite{pes85,pes93}. In particular, Pestov described
its topology.\\
Let $(X,\mathcal U)$ be a non-archimedean uniform space.   \ben \item For $\Omega =\mathbf{NA}$ we obtain the {\it free non-archimedean group}
$F_{\sna}$. \item In case $\Omega =\mathbf{NA_b}$, the universal object is the {\it free non-archimedean balanced group}
$F^b_{\sna}$.\een  These groups were defined and studied by Megrelishvili and the author in \cite{MS13}. \vskip 0.3cm We collect some known results from \cite{MS13, pes85}. Denote by $j_2$ the mapping $(x,y)\mapsto x^{-1}y$ from $X^{2}$
to  $F(X)$ and by $j_{2}^{\ast}$ the mapping
$(x,y)\mapsto xy^{-1}.$ 

\begin{defi}  \cite[Chapter 4]{DR81}	If $P$ is a group and $(V_n)_{n\in \N}$ a sequence of subsets of
	$P,$ define
	$$[(V_n)]:=\bigcup_{n\in \N}
	\bigcup_{\pi\in S_n} V_{\pi(1)} V_{\pi(2)}\cdots  V_{\pi(n)}.$$%
\end{defi}
\begin{remark} \cite[Remark 4.3]{MS13}\label{rem:const}
	Note that if $(V_n)_{n\in \N}$ is a constant sequence such that
	$$V_1=V_2=\cdots=V_n=\cdots =V,$$   then $[(V_n)]=\bigcup_{n\in
		\N}V^{n}.$ In this case we write $[V]$ instead of $[(V_n)].$
	It is easy to see that if $V=V^{-1},$ then $[V]$ is simply the subgroup generated by $V.$
\end{remark} \vskip 0.5 cm
\begin{defi} \
\ben \item \cite{pes85} For  every $\psi\in {\mathcal U}^{F(X)}$
let $$V_{\psi}:=\bigcup_{w\in F(X)}w(j_{2}(\psi(w))\cup
j_{2}^{\ast}(\psi(w)))w^{-1}.$$
\item  \cite[Definition 4.9.2]{MS13} As a particular case in which every $\psi$ is a constant
function we obtain the  set
$$\tilde{\eps}:=\bigcup_{w\in F(X)}w(j_{2}(\eps)\cup
j_{2}^{\ast}(\eps))w^{-1}.$$ \een
\end{defi}

\begin{thm}\label{thm:desc} \
\ben \item \textnormal{(Pestov \cite[Theorem 2]{pes85})} Let $(X,U)$ be a uniform space. The set
$\{[(V_{\psi_{n}})]\},$ where $\{\psi_{n}\}$ extends over the family of
all possible sequences of elements from ${\mathcal U}^{F(X)},$ is a
local base at the identity element of the uniform free topological group $F(X,\mathcal U).$
\item  \cite[Theorem 4.13]{MS13} Assume that $(X,U)$ is a non-archimedean uniform space. Then,
\ben[(a)] \item  The set $\{[\mathcal{V}_{\psi}]:\ \psi\in
\mathcal{U}^{F(X)}\}$  is a local base at the identity element of $F {\sna}(X,\mathcal U),$ the uniform free non-archimedean  group.
\item The family (of normal subgroups)
$\{[\tilde{\eps}]:\ \eps\in  \mathcal{B}\}$ is a local base at the identity element of   $F^b_ {\sna}(X,\mathcal U),$ the uniform free  non-archimedean balanced group.
\een \een
\end{thm}

\begin{remark} \label{rem:nar} Let $(X,\mathcal U)$ be a non-archimedean uniform space. By the universal properties of the universal objects it is clear that: \ben \item the topology of $F {\sna}(X,\mathcal U)$ is the maximal non-archimedean group topology on $F(X)$ that is  coarser than the topology of $F(X,\mathcal U).$ \item the topology of  $F^b_ {\sna}(X,\mathcal U)$  is the maximal non-archimedean balanced group topology on $F(X)$ that is coarser than the topology of $F(X,\mathcal U).$\een In particular,  if $F(X,\mathcal U)$ is non-archimedean, then $F(X,\mathcal U)$ coincides with $F {\sna}(X,\mathcal U).$ If $F(X,\mathcal U)$ is also balanced, then these groups coincide also with  $F^b_ {\sna}(X,\mathcal U).$
		\end{remark}
		\begin{thm}\label{thm:coin}
Let $(X,\mathcal U)$ be a uniform space. Suppose that there exists an infinite cardinal $\tau$ such that $\bigcap_{i\in I}\eps_i \in \mathcal U$ for any family of entourages $\{\eps_i:\ i\in I\}\subseteq \mathcal U$ with $|I|\leq \tau.$ Then, \ben \item the intersection of  any family of cardinality at most $ \tau$ of open subsets of $F(X,\mathcal U)$ is open. In particular,  $F(X,\mathcal U)$ is a $P$-group and we have $F(X,\mathcal U)=F_ {\sna}(X,\mathcal U).$ \item if the uniform space $(X,\mathcal U)$ is also $\tau$-narrow then $$F(X,\mathcal U)=F_ {\sna}(X,\mathcal U)=F^b_ {\sna}(X,\mathcal U).$$\een
		\end{thm}
\begin{proof}
$(1):$ Let $I$ be an arbitrary set  with $|I|\leq \tau.$ For every $i\in I$	let $\{\psi^i_n\}$  be a sequence of elements from $\mathcal U^{F(X)}$ (see Theorem \ref{thm:desc}.1). We define a function $\varphi$ as follows. For every $w\in F(X)$ let
	$\varphi(w)=\bigcap_{i\in I, n\in \N}\psi^i_n(w).$ By our assumption on the cardinal $\tau,$ we have $\varphi\in \mathcal U^{F(X)}.$ \\Clearly, $\mathcal{V}_{\varphi}\subseteq V_{\psi^i_{n}} \  \forall i\in I, \forall n\in \N.$ It follows that $[\mathcal{V}_{\varphi}]\subseteq \bigcap_{i\in I}[(V_{\psi^i_{n}})].$ By \cite[Theorem 2]{pes85} (see also Theorem \ref{thm:desc}.1) and Remark \ref{rem:const},  $[\mathcal{V}_{\varphi}]$ is a neighborhood of the identity of $F(X,\mathcal U)$. It follows that   the intersection of  any family of cardinality at most $\tau$ of open subsets of $F(X,\mathcal U)$ is open. Therefore, $F(X,\mathcal U)$ is a $P$-group. By Lemma \ref{lem:na} and Remark \ref{rem:nar}, $F(X,\mathcal U)=F_ {\sna}(X,\mathcal U).$\\
	$(2)$: It is known that the universal morphism  $i:(X,\mathcal U)\to F(X,\mathcal U)$ is a uniform embedding and that $i(X)$ algebraically generates  $F(X,\mathcal U).$
	Since $(X,\mathcal U)$ is $\tau$-narrow, we obtain by \cite[Theorem 5.1.19]{AT} (see also \cite[Exercise 5.1.a]{AT}) that $F(X,\mathcal U)$ is $\tau$-bounded. By item $(1)$ and Lemma
	\ref{lem:bala}, we conclude that the non-archimedean group $F(X,\mathcal U)$ is also balanced and so we have $F(X,\mathcal U)=F_ {\sna}(X,\mathcal U)=F^b_ {\sna}(X,\mathcal U).$
\end{proof}	

Omitting the $\tau$-narrowness assumption from Theorem \ref{thm:coin}.2, we obtain the following counterexample.
\begin{example}
	By \cite[Example 3.14]{NT}, for every cardinal $\tau>\aleph_1,$ there exists a Hausdorff uniform $P$-space such that $w(X,\mathcal U)=\aleph_1<\tau<\chi(F(X,\mathcal U)).$ In view of Theorem \ref{thm:coin}.1  and \cite[Theorem 4.16.1]{MS13}, we have $F(X,\mathcal U)=F_ {\sna}(X,\mathcal U)\neq F^b_ {\sna}(X,\mathcal U).$
\end{example}

As corollaries we obtain the following two results of Nickolas and Tkachenko.
\begin{corol}
\cite[Lemma 3.12]{NT} If $(X,\mathcal U)$ is an $\aleph_0$-narrow uniform $P$-space, then the group $F(X,\mathcal U)$ has a base at the identity consisting of open normal subgroups.  
\end{corol}	
\begin{proof}
By Theorem \ref{thm:coin}.2,  $F(X,\mathcal U)=F^b_ {\sna}(X,\mathcal U).$  Now use item $(b)$ of Theorem \ref{thm:desc}.2.
\end{proof}
\begin{corol}\cite[Theorem 3.13]{NT} If $(X,\mathcal U)$ is an $\aleph_0$-narrow uniform $P$-space, then $\chi(F(X,\mathcal U))=w(X,\mathcal U).$
\end{corol}
\begin{proof}
If $(X,\mathcal U)$ is an $\aleph_0$-narrow uniform $P$-space, then by Theorem \ref{thm:coin}.2 we have $F(X,\mathcal U)=F^b_ {\sna}(X,\mathcal U).$ On the other hand, by \cite[Theorem 4.16.1]{MS13},  we have $\chi (F^b_ {\sna}(X,\mathcal U))=w(X,\mathcal U).$  We conclude that $\chi(F(X,\mathcal U))=w(X,\mathcal U).$\end{proof}
\section{The subsets $B_n$}\label{sec:last}
\begin{thm}\label{thm:bn}
 Suppose that $\tau$ is a topological $P$-group topology on a free group $F(X)$  and $\mu$ is  either the right, left or  two-sided uniformity of $(F(X),\tau).$ Then, $\eps\in \mu$ if (and only if) $\eps\cap (B_n\times B_n)\in \mu \cap (B_n\times B_n) \ \forall n\in \N.$
\end{thm}
\begin{proof}
It is clear that if $(F(X),\tau)$ is a $P$-group, then $(F(X),\mu)$ is a uniform $P$-space.
Let us assume that $\eps\cap (B_n\times B_n)\in \mu \cap (B_n\times B_n) \ \forall n\in \N.$ Then, for every $n\in \N$ there exists $\delta_n\in \mathcal U$ such that  $\eps\cap (B_n\times B_n)=\delta_n\cap (B_n\times B_n).$ Since $(F(X),\mu)$ is a uniform $P$-space, we have $\delta=\cap_{n\in \N}\delta_n\in \mu.$ Hence, $$\delta= \bigcup_{n\in \N}(\delta \cap (B_n\times B_n))\subseteq \bigcup_{n\in \N}(\delta_n \cap (B_n\times B_n))=\bigcup_{n\in \N}(\eps\cap (B_n\times B_n))=\eps,$$  and we conclude that $\eps\in \mu.$
\end{proof}
\begin{corol}\label{cor:bns}
	Let $(X,\mathcal U)$ be a uniform $P$-space and $\mu$ be either the right, left or two-sided uniformity of  $F(X,\mathcal U).$ Then, $\eps\in \mu$ if (and only if) $$\eps\cap (B_n\times B_n)\in \mu \cap (B_n\times B_n) \ \forall n\in \N.$$
\end{corol}
\begin{proof}
By Theorem \ref{thm:coin},  $F(X,\mathcal U)$	 is a $P$-group. Now the proof follows from Theorem \ref{thm:bn}.
\end{proof}
\begin{corol}
Let $(X,\mathcal U)$ be a uniform $P$-space. Then, $F(X,\mathcal U)$ is balanced if and only if $B_n$ is balanced for every $n\in \N.$
\end{corol}
\begin{thm}\label{thm:last}
Let $(X,\mathcal U)$ be a uniform $P$-space. The following are equivalent:
\ben 
\item  $B_n$ is strongly functionally balanced for every $n\in \N.$
\item $B_n$ is  functionally balanced for every $n\in \N.$
\item $F(X,\mathcal U)$ is functionally balanced.
\item $F(X,\mathcal U)$ is strongly functionally balanced.
\een
\end{thm}
\begin{proof}
$G:=F(X,\mathcal U)$ is a $P$-group by Theorem \ref{thm:coin}. So, the implications $(1)\iff (2)$ and $(3)\iff (4)$ can be derived from  Theorem \ref{thm:equiv}. \\
$(2)\Rightarrow (3):$  Using Theorem \ref{thm:equiv}, it suffices to show that 
if   $\eps\in \mathcal L_G$ with at most
$\mathfrak c$ equivalence classes, then $\eps\in \mathcal R_G.$ For such an $\eps$ it is clear that $\eps_n:=\eps\cap (B_n\times B_n)$ has at most
$\mathfrak c$ equivalence classes in $B_n.$ It follows from our assumption $(2)$ and from  Theorem \ref{thm:equiv} (with $B=B_n$) that $\eps_n\in \mathcal R_G^{n}.$  Corollary \ref{cor:bns} implies that $\eps\in \mathcal R_G,$ as needed.\\
$(3)\Rightarrow (2):$ Suppose that $F(X,\mathcal U)$ is functionally balanced and fix an arbitrary $n\in \N.$ We  show that $B_n$ is functionally balanced. 
Let $\eps_n:=\eps\cap (B_n\times B_n)\in \mathcal L^n_G$ be an equivalence relation with  at most
$\mathfrak c$ equivalence classes in $B_n,$ where $\eps\in \mathcal L_G$ is  an equivalence relation on $F(X).$ Let $$\delta:= \eps\cup ((F(X)\setminus \eps[B_n])\times (F(X)\setminus \eps[B_n])).$$ We will show that $\delta$ has the following properties:\ben[(a)] \item
$\delta$ is an equivalence relation with $\eps\subseteq \delta.$
\item $\delta_n=\delta\cap (B_n\times B_n)=\eps_n.$
\item $\delta$ has at most $\mathfrak c$ equivalence classes in $F(X).$
\een
We prove the nontrivial part of $(a)$. Namely, the transitivity of $\delta.$
Let $(x,y),(y,z)\in \delta.$ If $(x,y), (y,z)\in \eps$ or $(x,y),(y,z)\in 
 (F(X)\setminus \eps[B_n])\times (F(X)\setminus \eps[B_n])$ the assertion is trivial. So, without loss of generality assume that $(x,y)\in \eps$ and $(y,z)\in 
 (F(X)\setminus \eps[B_n])\times (F(X)\setminus \eps[B_n]).$  Since $(x,y)\in \eps$ and $y\in (F(X)\setminus \eps[B_n]),$ it follows that $x\in (F(X)\setminus \eps[B_n]).$ Since $z$ is also in $(F(X)\setminus \eps[B_n]),$ we have $(x,z)\in  (F(X)\setminus \eps[B_n])\times (F(X)\setminus \eps[B_n])\subseteq \delta.$\\
 To see that property $(b)$ is satisfied, first observe that $B_n\subseteq \eps[B_n].$ Therefore, $$\delta_n=\delta\cap (B_n\times B_n)= (\eps\cup ((F(X)\setminus \eps[B_n])\times (F(X)\setminus \eps[B_n])))\cap (B_n\times B_n)=$$$$=(\eps\cap (B_n\times B_n))\cup ((F(X)\setminus \eps[B_n])\times (F(X)\setminus \eps[B_n]))\cap (B_n\times B_n))=$$$$=(\eps\cap (B_n\times B_n))\cup \emptyset= \eps_n.$$\\
 To prove $(c)$ combine the following two observations. On the one hand, the fact that $\eps_n$ has at most
 $\mathfrak c$ equivalence classes in $B_n,$ implies that  there exist at most $\mathfrak c$ equivalence relations $\delta[x],$ where $x\in \eps[B_n].$ On the other hand,  by the definition of $\delta,$ the number of equivalence relations $\delta[x],$ with $x\notin   \eps[B_n]$ is less than two. So, $\delta$ has at most $\mathfrak c$ equivalence classes in $F(X).$\\
 Now, using $(a)$ and $(c)$ together with our assumption that $F(X,\mathcal U)$ is functionally balanced, we obtain by Theorem \ref{thm:equiv} that 
 $\delta\in  \mathcal R_G.$ Finally, we use property $(b)$ and Corollary \ref{cor:bns}
to conclude that $\eps_n\in \mathcal R_G^n.$
\end{proof}
\noindent \textbf{Acknowledgments:} I would like to thank A.M. Brodsky,  M.
Megrelishvili, L. Jennings and the referee for  their useful suggestions. This work is   supported by Programma SIR 2014 by MIUR, project
GADYGR, number RBSI14V2LI, cup G22I15000160008, and also by a grant of Israel Science Foundation (ISF 668/13).
\bibliographystyle{amsalpha}

\end{document}